\documentclass[final]{siamltex}

\usepackage{hyperref}

\usepackage{xspace}
\usepackage{subfig}
\usepackage{mathrsfs,stmaryrd}
\usepackage{url}
\usepackage{xcolor}
\usepackage{enumitem}
\usepackage{tikz}
\usetikzlibrary{decorations.pathreplacing}
\usetikzlibrary{patterns}
\usepackage{amssymb,amsmath}

\newcommand\xlp{$L$-periodic w.r.t. $x$}
\newcommand\Ax{A^x}
\newcommand\Bx{B^x}
\newcommand\Rx{R^x}

\newcommand{\kalm}[2]{ \left[\, #1 \, : \, #2 \, \right] }

\newcommand\gln{\mathrm{GL}_n(\R)}

\newcommand{\field}[1]{\ensuremath{\mathbb{#1}}}

\newcommand{\R}{\field{R}\xspace}

\newcommand{\N}{\field{N}\xspace}

\newcommand\inter[1]{\llbracket #1\rrbracket}

\newcommand\floor[1]{\lfloor#1\rfloor}
\newcommand\ceil[1]{\lceil #1\rceil}

\newcommand{\cont}{\Theta}

\newcommand{\opD}{\mathscr{D}}
\newcommand{\opL}{\mathscr{L}}
\newcommand{\opM}{\mathscr{M}}
\newcommand{\setA}{\mathcal{A}}
\newcommand{\setR}{\mathcal{R}}
\newcommand{\setB}{\mathcal{B}}

\newcommand\kmin{k_\mathrm{min}}
\newcommand\kmax{k_\mathrm{max}}

\newcommand{\eqdef}{\overset{\mathrm{def}}{=}}

\newcommand\Id{\mathrm{Id}}

\newcommand\ddt{\frac{d}{dt}}

\newcommand\car{X}
\newcommand\excar{\ext{\car}}

\newcommand\ball{\mathrm{B}}

\newcommand\ext[1]{\overline{#1}}

\newcommand\Tau{\mathcal{T}}

\renewcommand\epsilon{\varepsilon}
\newcommand{\rpourregularite}{r}
\newcommand{\spourregularite}{s}

\def\ds{\displaystyle}
\def\norm#1{\|#1\|}

\newtheorem{remark}{Remark}

\newcommand{\card}[1]{\,\mathrm{card}\,#1}

\newcommand{\abs}[1]{ \left|#1\right| }

\newcommand{\ens}[1]{ \left\{#1\right\} }

\newcommand{\adh}[1]{ \overline{#1} }

\title{Internal controllability of first order quasilinear hyperbolic systems with a reduced number of controls}

\author{Fatiha Alabau-Boussouira, Jean-Michel Coron and Guillaume Olive\thanks{Jean-Michel Coron and Guillaume Olive were supported by the ERC advanced grant 266907 (CPDENL) of the 7th Research Framework Programme (FP7)} }

\begin{document}

\maketitle

\begin{abstract}
In this paper we investigate the exact controllability of $n \times n$ first order one-dimensional quasilinear hyperbolic systems by $m<n$ internal controls that are localized in space in some part of the domain.
We distinguish two situations.
The first one is when the equations of the system have the same speed.
In this case, we can use the method of characteristics and obtain a simple and complete characterization for linear systems.
Thanks to a linear test this also provides some sufficient conditions for the local exact controllability around the trajectories of semilinear systems.
However,  when the speed of the equations are not anymore the same, we see that we encounter the problem of loss of derivatives if we try to control quasilinear systems with a reduced number of controls.
To solve this problem, as in a prior article by J.-M. Coron and P. Lissy on a Navier-Stokes control system, we first use  the notion of algebraic solvability due to M. Gromov. However, in contrast with this prior article where a standard fixed point argument
could be used to treat the nonlinearities,  we use here a fixed point theorem of Nash-Moser type due to M. Gromov in order to handle the problem of loss of derivatives.
\end{abstract}



\pagestyle{myheadings}
\thispagestyle{plain}
\markboth{F.~Alabau-Boussouira, J.-M.~Coron and G.~Olive}{Controllability of first order hyperbolic systems}

\section{Introduction}

In this paper we investigate the exact controllability of $n \times n$ first order one-dimensional quasilinear hyperbolic systems by $m<n$ internal controls that are localized in space in some part of the domain.
While the controllability of quasilinear hyperbolic systems by boundary controls has been intensively studied, \cite{Cir,LR2002,LR2003,Wan,Zha,LRW}, to our knowledge there are no equivalent results for the internal controllability.
On the other hand, the controllability of systems of PDEs with a reduced number of controls has been a challenging problem for the last decades, see for instance \cite{Ala0,Ala1} for the first results on linear hyperbolic systems (see also \cite{Dag,ABL,Ala2,DLL,Ala}) and \cite{Zha} for quasilinear hyperbolic systems, the survey \cite{AKBGBdTsurvey} (and the references inside \cite{dT}, \cite{GBPG}, \cite{2007-Guerrero-SICON}, etc.) for linear parabolic systems, \cite{CGR} for a nonlinear parabolic system, \cite{2009-Coron-Guerrero-JDE} for Stokes equations, \cite{2006-Fernandez-Cara-Guerrero-Immanuvilov-Puel-SICON}, 
\cite{2009-Coron-Guerrero-JMPA} and \cite{CL} for Navier-Stokes equations.
Let us also point out that, in many of these articles, the general strategy is to start with a controllability result in the case where there are as many controls as the number of equations and then to try to remove some of these controls by a suitable procedure. In the present article we will follow this general strategy by making use of the so-called fictitious control method introduced  in \cite{1992-JMC-MCSS} for the control of linear ordinary differential equations and in \cite{GBPG} for the control of linear partial differential equations, an article where this illuminating terminology was moreover introduced.

In \cite{LR2003}, the authors introduced a constructive method to control quasilinear systems of $n$ equations by $n$ boundary controls.
This proficient method is based on existence and uniqueness results of semi-global solutions \cite{LJ} (i.e. with large time and small data) that they apply to several mixed initial-boundary value problems, using also the equivalent roles of the time and the space.
As we shall see below, using a method of extension of the domain (as it is often used in the parabolic framework), we can recover this result for the internal controllability, that is we can prove the controllability of $n \times n$ quasilinear systems by $n$ internal controls.
The situation is more complicated when we have less controls than equations.
Indeed the extension method is not anymore applicable in this context.
Thus, we need to develop direct methods to solve the problem of internal controllability.

We start the study with linear systems of equations with the same velocity.
In this case, we can apply the method of characteristics and obtain a complete and simple characterization of the exact controllability.
We show that the linear system can be viewed as a parametrized family of ODEs that are controlled independently.
The difficulty is actually to prove that this is enough to build a smooth ($C^1$) control for the linear hyperbolic system.
Moreover, since we look for controls of the hyperbolic system that are localized in some part of the domain, a nonstandard condition on the supports of the ODEs also appears and needs to be handled.
Another key point of the proof is the explicit formula of the HUM control for ODEs. \phantom{\cite{BFH}}

Using then a standard fixed point argument we can obtain sufficient conditions for the local exact controllability around the trajectories of semilinear systems.
However, when the equations do not have the same speed anymore and the nonlinearity is stronger, that is when we consider quasilinear systems, the standard linear test fails because of a loss of derivatives.
To solve this problem, we need to use a fixed point of Nash-Moser type.
We propose to use the fixed point theorem of M.~Gromov \cite[Section 2.3.2, Main Theorem]{Gro}, which is based on the notion of algebraic solvability for partial differential operators (see Definition \ref{def alg} below for more details).
The method consists in, first controlling the $n \times n$ system by $n$ controls, and then to eliminate a certain number of controls through the algebraic solvability.
The use of the Gromov algebraic solvability in the framework of the control theory was introduced in \cite[Pages 13-15]{Cor} for the control of linear ordinary differential equations (however it does not lead to new results in this case), in \cite{CL} for a Navier-Stokes control system, and in \cite{DL15,DL16} for some first order coupled parabolic systems.
In these works, the parabolicity allows to have smooth controls, and thus to avoid the problem of loss of derivatives.
The difference between the present work and \cite{CL}, where the algebraic solvability was the difficult task (the fixed point was standard), is that, following the algebraic solvability step, we show how to apply the fixed point theorem of Gromov to obtain the controllability of the quasilinear system.
Last, but not least, this method is probably not optimal with the regularity obtained, which leaves some challenging problems.

\section{Systems of equations with the same velocity}

\subsection{Linear systems}\label{sec same speed lin}

Let us consider the following linear hyperbolic system with periodic boundary conditions:
\begin{equation}\label{syst same speed}
\left\{
\begin{array}{l}
y_t + y_x + A(t,x)y= B(t,x) \cont, \quad (t,x) \in [0,T]\times[0,L], \\
y(t,L)=y(t,0), \quad t \in [0,T], \\
y(0,x)= y^0(x), \quad x \in [0,L].
\end{array}
\right.
\end{equation}
In \eqref{syst same speed}, $T>0$ is the control time, $L>0$ is the length of the domain.
$A$ and $B$ are time and space dependent matrices of size $n \times n$ and $n \times m$, respectively, where $n \in \N^*$ denotes the number of equations of the system and $m \in \N^*$ the number of controls (with possibly $m<n$).
$y^0$ is the initial data and $y(t,\cdot):[0,L] \longrightarrow \R^n$ is the state at time $t \in [0,T]$.
Finally, $\cont(t,\cdot):[0,L] \longrightarrow \R^m$ is the distributed control at time $t \in [0,T]$, that we look subject to the constraint
\begin{equation}\label{constraint}
\supp \cont \subset [0,T]\times [a,b],
\end{equation}
where here, and in what follows, the interval $[a,b]$, with $0 \leq a < b \leq L$, is fixed.

Throughout this article, for $k \in \N$ and $p \in \N^*$, we denote by $C^k_L([0,T]\times[0,L])^p$ (\textit{resp.} $C^k_L([0,L])^p$) the Banach space of functions $y \in C^k([0,T]\times[0,L])^p$ (\textit{resp.} $y \in C^k([0,L])^p$) that are \xlp, that is
\begin{subequations}
\begin{gather}
\partial^i_x y(t,0)=\partial^i_x y(t,L), \quad \forall t \in [0,T],\quad \forall i \in \inter{0,k}. \\
\left(
\textit{resp. }
y^{(i)}(0)=y^{(i)}(L), \quad \forall i \in \inter{0,k}.
\right)
\end{gather}
\end{subequations}

All along Section \ref{sec same speed lin} we assume that $A \in C^1_L([0,T]\times[0,L])^{n \times n}$, $B \in C^1_L([0,T]\times[0,L])^{n \times m}$.
These assumptions are made for regularity purposes, see below.

We recall that, for every $T>0$, there exists $C>0$ such that, for every $\cont \in C^1_L([0,T]\times[0,L])^m$ and every $y^0 \in C^1_L([0,L])^n$,
there exists a unique classical global solution $y \in C^1_L([0,T]\times[0,L])^n$ to \eqref{syst same speed}, and this solution satisfies the estimate
$$\norm{y}_{C^1} \leq C \left(\norm{y^0}_{C^1}+\norm{\cont}_{C^1}\right).$$
This well-posedness result follows from the classical theory of linear hyperbolic systems using the method of characteristics \cite{LY}.
Note that the kind of boundary conditions we consider is nonlocal but, as already noticed in \cite{CBAN} (see also \cite{LRW}), it can always be reduced to more standard (i.e. local) boundary conditions by introducing the enlarged system satisfied by $(y,\tilde{y})$ where $\tilde{y}(t,x)=y(t,L-x)$.

\begin{definition}
We say that System \eqref{syst same speed} is exactly controllable at time $T>0$ if, for every $y^0 \in C^1_L([0,L])^n$ and for every $y^1 \in C^1_L([0,L])^n$,
there exists a control $\cont \in C^1_L([0,T]\times[0,L])^m$ that satisfies the constraint \eqref{constraint} and is such that the corresponding solution $y \in C^1_L([0,T]\times[0,L])^n$ to \eqref{syst same speed} satisfies
$$y(T,x)=y^1(x), \quad \forall x \in [0,L].$$
\end{definition}

\subsubsection{The extended characteristics}

Let us now introduce an important tool when dealing with hyperbolic systems, namely the characteristics of the system.
In our case (speed 1 on each equation), the characteristic $\car$ of System \eqref{syst same speed} passing through the point $(t_0,x_0) \in [0,T]\times[0,L]$ is the straight line
$$\car(t,t_0,x_0) \eqdef t-t_0+x_0, \quad t \in [0,T].$$
However, in this paper, the crucial tools we need are the \textit{extended characteristics} $\excar: [0,T]\times[0,L) \longrightarrow [0,L]$, defined by (see Fig. \ref{fig caract} below):
$$
\excar(t,x) \eqdef 
\left\{\begin{array}{l}
\car\left(t,0,x\right) \, \mbox{ if } t \in \left[0,\tau_0(x,L)\right], \\
\car\left(t,\tau_{k-1}(x,L),0\right) \, \mbox{ if } t \in \left(\tau_k(x,0),\tau_k(x,L)\right], \, k \in \inter{\kmin(x,0),\kmax(x,0)},
\end{array}\right.
$$
where, for every $k \in \N$ and $c \in [0,L]$, we introduce the functions
$$\tau_k(x,c) \eqdef
\left\{\begin{array}{cl}
0 & \ds \mbox{ if }  c-x+kL\in (-\infty,0), \\
c-x+kL &\ds \mbox{ if } c-x+kL \in [0,T], \\
T & \ds \mbox{ if } c-x+kL \in (T,+\infty),
\end{array}\right.
$$
and $\kmin(x,c)$ (\textit{resp.} $\kmax(x,c)$) denotes the smallest (\textit{resp.} greatest) integer $k \in \N$ such that $c-x+kL>0$ (\textit{resp.} $c-x+kL<T$).
More precisely, denoting by $\floor{\cdot}$ the floor function and $\ceil{\cdot}$ the ceiling function,
$$
\kmin(x,c) \eqdef \floor{\frac{-c+x}{L}}+1
=\left\{
\begin{array}{cl}
0 & \mbox{ if } x \in [0,c), \\
1 & \mbox{ if } x \in [c,L),
\end{array}\right.
$$
$$
\kmax(x,c) \eqdef \ceil{\frac{T-c+x}{L}}-1
=\left\{
\begin{array}{cl}
\ceil{\frac{T-c}{L}}-1 &\ds \mbox{ if } x \in \left[0,p(c)\right], \\
\ceil{\frac{T-c}{L}} &\ds \mbox{ if } x \in \left(p(c),L\right),
\end{array}\right.
$$
where
$$p(c) \eqdef \left(\ceil{\frac{T-c}{L}}-\frac{T-c}{L}\right)L.$$
Note that $\tau_k(x,0)=\tau_{k-1}(x,L)$ for every $k \geq 1$ and $\tau_{\kmax(x,0)}(x,L)=T$, so that $\excar(t,x)$ is indeed defined for every $t \in [0,T]$.

\begin{figure}[h!]
\centering
\begin{tikzpicture}[scale=0.3]

\draw [thick] [<->] (0,15) node [left] {$x$} -- (0,0)  -- (28,0) node [below right] {$t$};
\draw (-0.4,-0.2) node [below] {$0$};
\draw (0,12) node [left] {$L$};
\draw (25,0.2)--(25,-0.2) node [below] {$T$};
\draw [thick,dashed,fill=gray!40] (0,3) rectangle (25,8);
\draw [thick] (0,0) rectangle (25,12);

\draw (0.2,3)--(-0.2,3) node [left] {$a$};
\draw [very thick] (0,3)--(25,3);
\draw (0.2,8)--(-0.2,8) node [left] {$b$};
\draw [very thick] (0,8)--(25,8);
\draw (0.2,9)--(-0.2,9) node [left] {$x$};

\draw (0,9)--(3,12);
\draw [dashed] (3,12)--(3,0);
\draw (3,0)--(15,12);
\draw [dashed] (15,12)--(15,0);
\draw (15,0)--(25,10);
\draw (9,10) node[scale=1.5] {$\excar(\cdot,x)$};

\draw (0,9) node [rotate=45,scale=1.7] {$[$};
\draw (3,12) node [rotate=45,scale=1.7] {$]$};
\draw (3,0) node [rotate=45,scale=1.7] {$($};
\draw (15,12) node [rotate=45,scale=1.7] {$]$};
\draw (15,0) node [rotate=45,scale=1.7] {$($};
\draw (25,10) node [rotate=45,scale=1.7] {$]$};

\draw (3,0.2)--(3,-0.2) node[below] {\small $\tau_0(x,L)$};
\draw (15,0.2)--(15,-0.2) node[below] {\small $\tau_1(x,L)$};

\draw [dashed] (6,3)--(6,0);
\draw (6,0.2)--(6,-0.2);
\draw (6,-1) node[below] {\small $\tau_1(x,a)$};
\draw [dashed] (11,8) -- (11,0);
\draw (11,0.2)--(11,-0.2);
\draw (11,-1) node[below] {\small $\tau_1(x,b)$};

\end{tikzpicture}
\caption{The extended characteristics $\excar(\cdot,x)$.}
\label{fig caract}
\end{figure}
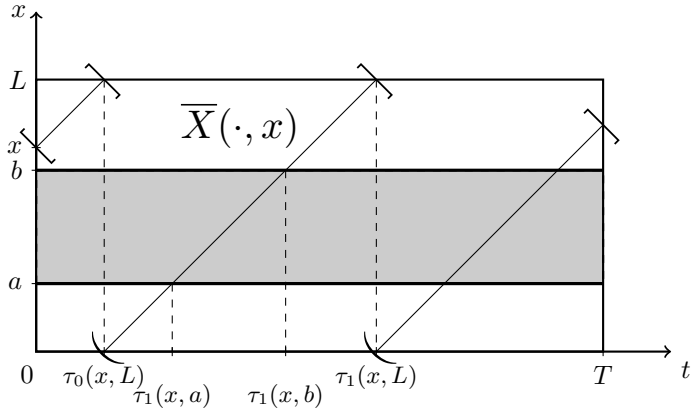

For $0 \leq a < b \leq L$, we list below some properties of these functions, under the essential assumption that every extended characteristic $\excar$ crosses the domain $[0,T]\times[a,b]$ at some time, that is
$$T>L-(b-a).$$
\begin{enumerate}
\item
$\kmin(x,b) \leq \kmax(x,a)$.
\item
$\tau_k(x,a)<\tau_k(x,b)$ for every $k \in \inter{\kmin(x,b),\kmax(x,a)}$ if $x \neq b$ and $x \neq p(a)$.
\item
$\tau_k(x,b)\leq\tau_{k+1}(x,a)$ for every $k \in \inter{\kmin(x,b),\kmax(x,a)-1}$.
\item
$\excar(t,x) \in (a,b)$ for every $t \in \left(\tau_k(x,a),\tau_k(x,b)\right)$ for every $k$ satisfying 
 $$ k \in \inter{\kmin(x,b),\kmax(x,a)}.$$
\end{enumerate}

We then introduce the following open sets (see Fig. \ref{fig trapezium} below)
$$
\Tau_0 \eqdef
\ens{(t,x) \in (0,T)\times(0,L) \quad \middle| \quad t \in \left(0,\tau_0(x,L)\right)},
$$
$$
\Tau_{\ceil{\frac{T}{L}}} \eqdef
\ens{(t,x) \in (0,T)\times \left(p(0),L\right), \quad \middle| \quad t \in \left(\tau_{\ceil{\frac{T}{L}}}(x,0),T\right)},
$$
and, for $k \in \inter{1,\ceil{\frac{T}{L}}-1}$,
$$
\Tau_k \eqdef
\ens{(t,x) \in (0,T)\times(0,L) \quad \middle| \quad t \in \left(\tau_k(x,0),\tau_k(x,L)\right)}.
$$
We set
$$\Tau \eqdef \bigcup_{k=0}^{\ceil{\frac{T}{L}}} \Tau_k.$$

\begin{figure}[h!]
\centering
\begin{tikzpicture}[scale=0.3]

\draw [thick] [<->] (0,15) node [left] {$x$} -- (0,0)  -- (28,0) node [below right] {$t$};
\draw (-0.4,-0.2) node [below] {$0$};
\draw (0,12) node [left] {$L$};
\draw (25,0.2)--(25,-0.2) node [below] {$T$};
\draw [thick] (0,0) rectangle (25,12);

\draw [dashed] (0,12)--(12,0);
\draw [dashed] (12,12)--(24,0);
\draw [dashed] (24,12)--(25,11);

\draw (3,3) node[scale=1.5] {$\Tau_0$};
\draw (13,5) node[scale=1.5] {$\Tau_1$};
\draw (23,7) node[scale=1.5] {$\Tau_2$};

\draw (6.3,8) node[scale=1] {$\adh{\Tau_0} \cap \adh{\Tau_1}$};
\draw (17,9.5) node[scale=1] {$\adh{\Tau_1} \cap \adh{\Tau_2}$};

\end{tikzpicture}
\caption{Domains $\Tau_k$ and the parts $\adh{\Tau_k} \cap \adh{\Tau}_{k+1}$ of their boundary.}
\label{fig trapezium}
\end{figure}

\begin{remark}\label{rem reg tools}
Let us give some comments about the properties of the extended characteristics.
First, we have
$$\excar \in C^1\left(\Tau\right),$$
with $\excar_t(t,x)=\excar_x(t,x)=1$ for every $(t,x) \in \Tau$.
Moreover, at the boundary $\partial\Tau$, we have
$$\forall (t_0,x_0) \in \adh{\Tau_k} \cap \adh{\Tau}_{k+1}, \quad
\lim_{\substack{(t,x) \to (t_0,x_0) \\ (t,x) \in \Tau_k}} \excar(t,x)=L,
\quad
\lim_{\substack{(t,x) \to (t_0,x_0) \\ (t,x) \in \Tau_{k+1}}} \excar(t,x)=0,$$
$\lim_{\substack{(t,x) \to (t_0,0) }} \excar(t,x)=\excar(t_0,0)$ for every $t_0 \in [0,T]$, and
$$
\forall t_0 \in [0,T], \quad \lim_{\substack{(t,x) \to (t_0,L) }} \excar(t,x)=
\left\{\begin{array}{cl}
\excar(t_0,0) & \mbox{ if } t_0 \in (0,T], \\
L & \mbox{ if } t_0=0.
\end{array}\right.
$$
In particular $\excar$ has a continuous extension (still denoted by $\excar$) to all the points of the boundary $(t,L)$, $t \in [0,T]$.

Second, the map $(t,x) \mapsto (t,\excar(t,x))$ is a $C^1$-diffeomorphism from $\Tau$ to $\Tau'=(0,T)\times(0,L) \backslash \ens{(t,\excar(t,0)) \, \middle| \, t \in [0,T]}$.
We denote by $(t,\excar^{-1}(t,x))$ its inverse.
\end{remark}

Finally, for every $x \in [0,L]$, we denote by $\Ax$ and $\Bx$ the values of $A$ and $B$ along the extended characteristic $\excar(\cdot,x)$:
$$\Ax(t) \eqdef A(t,\excar(t,x)), \quad \Bx(t) \eqdef B(t,\excar(t,x)), \quad \forall t \in [0,T].$$

Clearly, $(t,x) \mapsto \Ax(t) \in C^1(\Tau)^{n \times n}$.
Since $A$ is \xlp, we have $A \in C^0\left(\ext{\Tau}\right)^{n \times n}=C^0([0,T]\times[0,L])^{n \times n}$.
On the same manner, since $\partial_t A$ and $\partial_x A$ are \xlp, we have $A \in C^1([0,T]\times[0,L])^{n \times n}$.
Note also that $A_L=A_0$.
The same statements hold for the map $(t,x) \mapsto \Bx(t)$ as well.

\subsubsection{Characterization of the controllability of \eqref{syst same speed}}

The main result of section \ref{sec same speed lin} is the following:
\begin{theorem}\label{thm same speed}
Let $T,L>0$ and $0 \leq a < b \leq L$.
System \eqref{syst same speed} is exactly controllable at time $T$ if, and only if, the following 2 conditions hold:
\begin{enumerate}[label={(${\cal H}_{\arabic*}$)}]
\item\label{hyp temps}
$T>L-(b-a)$.
\item\label{hyp caract}
For every $x \in [0,L)$, the following ODE is controllable:
\begin{equation}\label{equ caract}
\left\{\begin{array}{l}
\ds \ddt z(t)=- \Ax(t) z(t)+\Bx(t)\psi(t), \quad \forall t \in [0,T], \\
z(0)=z^0 \in \R^n,
\end{array}\right.
\end{equation}
with controls $\psi \in C^1([0,T])^m$ such that
\begin{equation}\label{constr ODE}
\psi \equiv 0 \mbox{ in } [0,T]\backslash \left(\bigcup_{k=\kmin(x,b)}^{\kmax(x,a)} \left[\tau_k(x,a),\tau_k(x,b)\right]\right).
\end{equation}
\end{enumerate}
\end{theorem}

\begin{remark}
When $[a,b]=[0,L]$, hypothesis \ref{hyp temps} and \eqref{constr ODE} are automatically satisfied.
\end{remark}

\begin{remark}
As we shall see below (Proposition \ref{prop ext Gramian}) the controllability of \eqref{equ caract} with \eqref{constr ODE} only depends on the values of $A$ and $B$ inside the control domain $[0,T]\times[a,b]$.
\end{remark}

\begin{remark}\label{rem estim control}
In the proof of Theorem \ref{thm same speed} we will explicitly construct a control $\cont$ that steers the solution $y$ to \eqref{syst same speed} from $y^0$ to $y^1$, see \eqref{def cont final} below.
We can see that this control $\cont$ satisfies the following additional properties:
\begin{enumerate}
\item
Continuity:
there exists $C>0$ (depending only on $T,L,a,b,A,B$) such that
$$\norm{\cont}_{C^1} \leq C \left(\norm{y^0}_{C^1}+\norm{y^1}_{C^1}\right).$$
\item
Locality:
there exists $\delta>0$ small enough (depending only on $T,L,a,b,A,B$), such that
\begin{equation}\label{reduc constraint}
\supp \cont \subset [\delta,T-\delta]\times[a+\delta,b-\delta].
\end{equation}
\item
Higher regularity:
if $y^0,y^1 \in C^k_L([0,L])^n$ and $A \in C^k_L([0,T]\times[0,L])^{n \times n}$, $B \in C^k_L([0,T]\times[0,L])^{m \times n}$ ($k \geq 1$), then
$$\cont \in C^k([0,T]\times[0,L])^m.$$
\end{enumerate}
\end{remark}

\subsubsection{Controllability of linear O.D.E with constraints}\label{sec cont ODE constr}

Let us recall that we know some powerful tools to characterize the controllability of linear time-varying ODEs if no constraint are imposed on the controls.
We state below the extensions of these theorems to the case where the controls are supported in some part of the domain.

Let us consider the $n \times n$ ODE
\begin{equation}\label{equ caract ex}
\left\{\begin{array}{l}
\ds \ddt z(t)=-A(t) z(t)+B(t)\psi(t), \quad \forall t \in [0,T], \\
z(0)=z^0 \in \R^n,
\end{array}\right.
\end{equation}
with $A \in C^1([0,T])^{n \times n}$, $B \in C^1([0,T])^{n \times m}$.
We want to characterize the controllability of \eqref{equ caract ex} with the following additional constraint on the controls:
\begin{equation}\label{constr ODE ex}
\psi \equiv 0 \mbox{ in } [0,T]\backslash \left(\bigcup_{i=1}^{M} \left[a_i,b_i\right]\right),
\end{equation}
where $0 \leq a_i<b_i \leq T$ are such that $b_i \leq a_{i+1}$ for every $i \in \inter{1,M-1}$.

Let us denote by $R \in C^1([0,T]\times[0,T])^{n \times n}$ the resolvent associated with $-A \in C^1([0,T])^{n \times n}$, that is, for every $s \in [0,T]$, $R(\cdot,s)$ is the classical solution to the ODE
$$
\left\{\begin{array}{l}
\ds \partial_t R(t,s)=-A (t)R(t,s), \quad \forall t \in [0,T], \\
R(s,s)=\Id.
\end{array}\right.
$$

\begin{proposition}\label{prop ext Gramian}
The ODE \eqref{equ caract ex} is controllable with \eqref{constr ODE ex} if, and only if, its controllability Gramian, that is the $n \times n$ matrix
\begin{equation}\label{def gram}
Q \eqdef \sum_{i=1}^{M}
\int_{a_i}^{b_i} R\left(T,s\right)B(s) B(s)^*R\left(T,s\right)^* \, ds,
\end{equation}
is invertible.
\end{proposition}

The proof of Proposition \ref{prop ext Gramian} can be adapted from the one of \cite[Theorem 5]{KHN}.
To do so, we consider the control problem \eqref{equ caract ex} with $\eta B$ instead of $B$, where $\eta$ is a cut-off function that vanishes outside $\bigcup_{i=1}^{M} \left[a_i,b_i\right]$ and is equal to $1$ in $\bigcup_{i=1}^{M} \left[a_i+\epsilon,b_i-\epsilon\right]$ with $\epsilon>0$ small enough so that, by continuity, the Gramian \eqref{def gram} with $a_i+\epsilon$ (\textit{resp.} $b_i-\epsilon$) instead of $a_i$ (\textit{resp.} $b_i$) remains invertible.

Thanks to the previous characterization, we obtain the following proposition (see \cite[Theorem 10]{KHN} and \cite{SM} for a proof):
\begin{proposition}\label{prop ext Kalman}
Assume that $A$ and $B$ are constant.
Then, the controllability of \eqref{equ caract ex}-\eqref{constr ODE ex} is equivalent to the algebraic condition
$$\rank \kalm{A}{B}=n,$$
where the $n \times nm$ matrix $\kalm{A}{B}$ is defined by
\begin{equation}\label{kalm matx}
\kalm{A}{B} \eqdef [B|A B|A^2 B | \cdots | A^{n-1}B].
\end{equation}
\end{proposition}

\begin{proposition}\label{prop ext SM}
Assume that
$$A \in C^{n-2}([0,T])^{n \times n} \text{ and }B \in C^{n-1}([0,T])^{n \times m}$$
and let us introduce the following notation:
$$
\forall t \in [0,T], \quad
\left\{\begin{array}{l}
\ds B_0(t)=B(t), \\
\ds B_j(t)=\ddt B_{j-1}(t)+A(t) B_{j-1}(t), \quad \forall j \in \inter{1,n-1},
\end{array}\right.
$$
and, for every $t \in [0,T]$, the $n \times nm$ matrix
$$\kalm{A}{B}(t)\eqdef[B_{0}(t)|B_{1}(t)|\cdots| B_{n-1}(t)],$$
(which provides an extension of \eqref{kalm matx}).
Then, the ODE \eqref{equ caract ex} is controllable with \eqref{constr ODE ex} if the following property holds:
$$\exists i \in \inter{1,M}, \quad \exists t_i \in [a_i,b_i], \quad \rank \kalm{A}{B}(t_i)=n.$$
\end{proposition}

\subsubsection{Proof of Theorem \ref{thm same speed}, sufficient part}

The proof of the sufficient part of Theorem \ref{thm same speed} relies on the following key lemma (the proof of which is postponed to the appendix; see Section~\ref{secproolkeylemma}).
It states that we can always reduce a little bit the domain of control.
This is a uniform result with respect to $x$ (compare with Proposition \ref{prop ext Gramian}).

All along this section, for $x \in [0,L]$ we denote by $\Rx \in C^1([0,T]\times[0,T])^{n \times n}$ the resolvent associated with $-\Ax \in C^1([0,T])^{n \times n}$.
It is important to notice that the map $(t,s,x) \mapsto \Rx(t,s)$ is of class $C^1([0,T]\times[0,T]\times[0,L])^{n \times n}$ since the map $(t,x) \mapsto \Ax(t)$ is also of class $C^1([0,T]\times[0,L])^{n \times n}$ (see, for instance \cite[Chapter V, Theorem 3.1]{Har}).

\begin{lemma}\label{key lemma}
Assume that \ref{hyp temps} and \ref{hyp caract} hold.
Then, there exists $\delta>0$ small enough and a cut-off function $\eta \in C^1([0,T]\times[0,L])$ with (see Fig. \ref{fig eta})
\begin{equation}\label{eta zero}
\eta \equiv 0 \mbox{ in } [0,T]\times [0,L] \backslash \left((\delta,T-\delta)\times(a+\delta,b-\delta)\right),
\end{equation}
such that, for every $x \in [0,L]$, the Gramian
\begin{equation}\label{def cont gram}
Q_x \eqdef \int_{0}^{T}  \Rx\left(T,s\right)\Bx(s) \Bx(s)^*\Rx\left(T,s\right)^* \eta(s,\excar(s,x)) \, ds,
\end{equation}
is invertible.
\end{lemma}

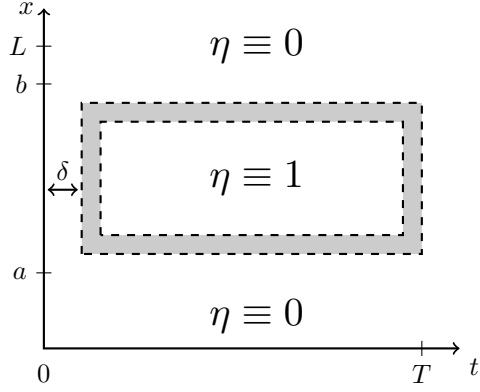
\begin{figure}[h!]
\centering
\begin{tikzpicture}[scale=0.5]

\draw [thick] [<->] (0,9) node [left] {$x$} -- (0,0) -- (11,0) node [below right] {$t$};
\draw (0,-0.2) node [below] {$0$};
\draw (0.2,8)--(-0.2,8) node [left] {$L$};
\draw (10,0.2)--(10,-0.2) node [below] {$T$};

\draw (0.2,2)--(-0.2,2) node [left] {$a$};
\draw (0.2,7)--(-0.2,7) node [left] {$b$};
\draw [thick,dashed,fill=gray!40] (1,2.5) rectangle (10,6.5);
\draw [thick,dashed,fill=white] (1.5,3) rectangle (9.5,6);

\draw [thick] [<->] (0.1,4.2)--(0.9,4.2);
\draw (0.5,4.7) node {$\delta$};

\draw (7.25,8) node [left,scale=1.5] {$\eta \equiv 0$};
\draw (7.25,4.5) node [left,scale=1.5] {$\eta \equiv 1$};
\draw (7.25,0.85) node [left,scale=1.5] {$\eta \equiv 0$};

\end{tikzpicture}
\caption{Reduction of the control domain.}
\label{fig eta}
\end{figure}

{\em Proof of Theorem \ref{thm same speed} (sufficient part).}
Assume that \ref{hyp temps} and \ref{hyp caract} hold and let us show that System \eqref{syst same speed} is exactly controllable at time $T$.
Let $y^0,y^1 \in C^1_L([0,L])^n$.

Let $Q_x$ be the controllability Gramian defined by \eqref{def cont gram}.
For every $(t,x) \in [0,T]\times[0,L]$, we set
\begin{equation}\label{def psi cont}
\psi(t,x) \eqdef
\eta(t,\excar(t,x))\Bx(t)^*\Rx(T,t)^* Q_x^{-1} \left(y^1\left(\excar(T,x)\right)-\Rx(T,0)y^0(x)\right).
\end{equation}
Since $\eta \in C^1_L([0,T]\times[0,L])$, we have $(t,x) \mapsto \eta(t,\excar(t,x)) \in C^1([0,T]\times[0,L])$.
Using Lebesgue's dominated convergence theorem, we obtain $x \mapsto Q_x \in C^1([0,L])^{n\times n}$.
On the other hand, since $y^1 \in C^1_L([0,L])^n$, we have $x \mapsto y^1(\excar(T,x)) \in C^1([0,L])^n$.
As a result,
$$\psi \in C^1([0,T]\times[0,L])^m.$$
For every $(t,x) \in [0,T]\times[0,L]$, we set
\begin{equation}\label{def cont final}
\cont(t,x)\eqdef
\left\{\begin{array}{cl}
\psi(t,\excar^{-1}(t,x)) & \mbox{ if }  (t,x) \in \Tau', \\
\psi(t,0) & \mbox{ if }  x=\excar(t,0), \\
0 & \mbox{ if } (t,x) \in \partial \left([0,T]\times[0,L]\right).
\end{array}
\right.
\end{equation}
 From \eqref{eta zero} we have $\cont \in C^1([0,T]\times[0,L])^m$ and \eqref{reduc constraint}.
Moreover,
$$\cont(t,\excar(t,x))=\psi(t,x), \quad \forall (t,x) \in [0,T]\times[0,L].$$
Let $y \in C^1([0,T]\times[0,L])^m$ be the solution to \eqref{syst same speed} associated with $\cont$ defined by \eqref{def cont final}.
For every $x \in [0,L)$, writing $y$ along the extended characteristics $\excar(\cdot,x)$, we see that $t \mapsto y(t,\excar(t,x))$ solves the ODE \eqref{equ caract} with $\psi$ defined by \eqref{def psi cont} and $z^0=y^0(x)$  (at least in the weak sense $W^{1,\infty}(0,T)^n$).
As a result we obtain $y(T,\excar(T,x))=y^1(\excar(T,x))$ for every $x \in [0,L)$.
Since $x \mapsto \excar(T,x)$ defines a bijective map from $[0,L)$ to $(0,L]$, we obtain that $y(T,x)=y^1(x)$ for every $x \in [0,L)$.
By continuity it follows that $y(T,x)=y^1(x)$ for every $x \in [0,L]$.
\endproof

\subsubsection{Proof of Theorem \ref{thm same speed}, necessary part}

Assume now that System \eqref{syst same speed} is exactly controllable at time $T$ and let us prove that this implies that \ref{hyp temps} and \ref{hyp caract} hold.

Assume first that $0<T \leq L-(b-a)$ and let $t_0=\max(0,T-a)$ and $x_0=-t_0+L$.
Note that $t_0 \in [0,T]$ and $x_0 \in [0,L]$.
Let $y^1=0$ and let $y^0 \in C^1_L([0,L])^n$ be such that
\begin{equation}\label{yz nn}
y^0(x_0) \neq 0.
\end{equation}
Writing $y$ along the characteristic $\car(s,t_0,0)$ for $s \in [t_0,T]$, gives
$$y(T,\car(T,t_0,0))=R_1(T,t_0)y(t_0,0)+ \int_{t_0}^T R_1(T,s)B(s)\cont(s,\car(s,t_0,0)) \, ds,$$
where $R_1$ is the resolvent associated with $t \in [t_0,T] \mapsto A(t,\car(t,t_0,0))$.
Now observe that, since $T \leq t_0+a$, we have $\car(s,t_0,0) \leq a$ for $s \in [t_0,T]$, so that, thanks to \eqref{constraint},
$$\cont(s,\car(s,t_0,0))=0, \quad \forall s \in [t_0,T].$$
As a result,
\begin{equation}\label{y in T}
y(T,\car(T,t_0,0))=R_1(T,t_0)y(t_0,0).
\end{equation}
Similarly, writing $y$ along the characteristic $\car(s,0,x_0)$ for $s \in [0,t_0]$, and using that $x_0 \geq b$, this leads to
$$y(t_0,L)=
y(t_0,\car(t_0,0,x_0))=R_2(t_0,0)y^0(x_0),$$
where $R_2$ is the resolvent associated with $t \in [0,t_0] \mapsto A(t,\car(t,0,x_0))$.
Since $y(t_0,0)=y(t_0,L)$, the previous equality can be combined with \eqref{y in T} and \eqref{yz nn} to show that $y(T,\car(T,t_0,0))\neq 0$ and therefore System \eqref{syst same speed} is not exactly controllable at time $T$.

We turn out to the necessity of \ref{hyp caract}.
Let $z^0,z^1 \in \R^n$ be fixed.
We then define $y^0 \equiv z^0$ and $y^1 \equiv z^1$, which belong to $C^1_L([0,L])^n$.
Thus, by assumption, there exists $\cont \in C^1_L([0,T]\times[0,L])^m$ that satisfies \eqref{constraint} such that the corresponding solution $y \in C^1([0,T]\times[0,L])^n$ to \eqref{syst same speed} satisfies
\begin{equation}\label{ytzero}
y(T,x)=z^1, \quad \forall x \in [0,L].
\end{equation}
For every $(t,x) \in [0,T]\times[0,L)$, we set
\begin{equation}\label{def psi cn}
\psi(t) \eqdef
\cont\left(t,\excar(t,x)\right).
\end{equation}
Since $\cont \in C^1_L([0,T]\times[0,L])^m$, we have $\psi \in C^1([0,T])^m$, while \eqref{constr ODE} follows from \eqref{constraint}.
Let $z \in C^1([0,T])^n$ be the solution to \eqref{equ caract} associated with $\psi$ defined by \eqref{def psi cn}.
Let $x \in [0,L)$ being fixed.
Writing $y$ along the extended characteristic $\excar(\cdot,x)$, we see that $t \in [0,T] \mapsto y(t,\excar(t,x))$ solves \eqref{equ caract} with $z^0=y^0(x)$ (at least in the weak sense $W^{1,\infty}(0,T)^n$).
By uniqueness of the solution to \eqref{equ caract} we then have
$$z(t)=y\left(t,\excar(t,x)\right), \quad \forall t \in [0,T].$$
In particular, $z(T)=z^1$ thanks to \eqref{ytzero}.
\endproof

\subsection{Semilinear systems}

Let us consider the following semilinear first order hyperbolic system with periodic boundary conditions:
\begin{equation}\label{syst semilin}
\left\{
\begin{array}{l}
y_t + y_x= f(y,\cont), \quad (t,x) \in [0,T]\times[0,L], \\
y(t,L)=y(t,0), \quad t \in [0,T], \\
y(0,x)= y^0(x), \quad x \in [0,L].
\end{array}
\right.
\end{equation}
We assume that the nonlinearity $f:\R^n \times \R^m \longrightarrow \R^n$ is of class $C^2$.

\begin{definition}
We say that $(\tilde{y},\tilde{\cont}) \in C^1_L([0,T]\times[0,L])^n \times C^1_L([0,T]\times[0,L])^m$ is a trajectory of System \eqref{syst semilin} if it is a classical solution to \eqref{syst semilin} for some $y^0 \in C^1_L([0,L])^n$ and if $\tilde{\cont}$ satisfies \eqref{constraint}.
\end{definition}

\begin{definition}
Let $(\tilde{y},\tilde{\cont})$ be trajectory of System \eqref{syst semilin}.
We say that System \eqref{syst semilin} is locally exactly controllable around the trajectory $(\tilde{y},\tilde{\cont})$ at time $T>0$ if, for every $\epsilon>0$, there exists $\mu>0$ such that, for every $y^0,y^1 \in C^1_L([0,T]\times[0,L])^n$ with
$$\norm{y^0-\tilde{y}(0,\cdot)}_{C^1} \leq \mu, \quad \norm{y^1-\tilde{y}(T,\cdot)}_{C^1} \leq \mu,$$
there exists a control $\cont \in C^1_L([0,T]\times[0,L])^m$ that satisfies \eqref{constraint} and a classical solution $y \in C^1_L([0,T]\times[0,L])^n$ to \eqref{syst semilin} such that
\begin{subequations}
\begin{gather}
y(T,x)=y^1(x), \quad \forall x \in [0,L],\\
\norm{y-\tilde{y}}_{C^1} \leq \epsilon, \label{estim def 2} \\
\norm{\cont-\tilde{\cont}}_{C^1} \leq \epsilon. \label{estim def 1}
\end{gather}
\end{subequations}
\end{definition}

Then, we have the following result.
The proof is classical and use the Banach fixed point theorem (see for instance \cite[Section 4.1]{Cor}).

\begin{theorem}\label{thm semilin}
Let $(\tilde{y},\tilde{\cont})$ be a trajectory of System \eqref{syst semilin}.
Assume that the linearization of system \eqref{syst semilin} around the trajectory $(\tilde{y},\tilde{\cont})$, that is the linear system
\begin{equation*}
\left\{
\begin{array}{l}
y_t + y_x= \ds \frac{\partial f}{\partial y} \left(\tilde{y}(t,x),\tilde{\cont}(t,x)\right) y + \frac{\partial f}{\partial \cont} \left(\tilde{y}(t,x),\tilde{\cont}(t,x)\right) \cont, \,\,(t,x) \in [0,T]\times[0,L], \\
y(t,L)=y(t,0), \quad t \in [0,T], \\
y(0,x)= y^0(x), \quad x \in [0,L],
\end{array}
\right.
\end{equation*}
is exactly controllable at time $T>0$.
Then, System \eqref{syst semilin} is locally exactly controllable around the trajectory $(\tilde{y},\tilde{\cont})$ at time $T$.
\end{theorem}

\section{Quasilinear systems with different velocities}

In what follows, we denote by $e_1=(1,0)$, $e_2=(0,1)$, the canonical basis of $\R^2$.

We are now interested in the controllability of the following $2 \times 2$ quasilinear system by one control force
\begin{equation}\label{syst quasilin}
\left\{
\begin{array}{l}
y_t + \Lambda(y)y_x + f(y)=e_1\cont, \quad (t,x) \in [0,T]\times[0,L], \\
y(t,L)=y(t,0), \quad t \in [0,T], \\
y(0,x)= y^0(x), \quad x \in [0,L],
\end{array}
\right.
\end{equation}
where
$$\Lambda(y)=\diag(\lambda_1(y),\lambda_2(y)), \quad \forall y \in \R^2,$$
with
\begin{equation}\label{ordre lambda}
\lambda_1(y)<\lambda_2(y) , \quad \lambda_1(y) \neq 0, \quad \lambda_2(y) \neq 0, \quad \forall y \in \R^2,
\end{equation}
and
$$f(y)= \begin{pmatrix}f_1(y)\\f_2(y)\end{pmatrix}, \quad \forall y \in \R^2,$$
with
$$f_1(0)=f_2(0)=0,$$
so that $(0,0)$ is a trajectory of system \eqref{syst quasilin}.
We assume that $\lambda_1,\lambda_2, f_1, f_2 \in C^{\infty}(\R^2)$.
In particular, System \eqref{syst quasilin} is hyperbolic (see for instance \cite[Pages 1-2]{LY}).

Then, for every $T>0$, there exist $C>0$ and $\mu>0$ such that, for every $\cont \in C^k_L([0,T]\times[0,L])$ and every $y^0 \in C^k_L([0,T]\times[0,L])^2$ ($k \in \N^*$) such that
$$\norm{\cont}_{C^k} \leq \mu, \quad \norm{y^0}_{C^k} \leq \mu,$$
there exists a unique semi-global classical solution $y \in C^k_L([0,T]\times[0,L])^2$ to \eqref{syst quasilin}, and this solution satisfies the estimate
$$\norm{y}_{C^k} \leq C \left(\norm{y^0}_{C^k}+\norm{\cont}_{C^k}\right).$$
We refer to \cite{LJ,Wan} for a proof of this well-posedness result.

The technical point in the method we will develop lies in the algebraic solvability (see Section \ref{sect alg} below).
Since the eigenvalues of $\Lambda(y)$ might be distinct, the more the number $n$ of equations of the system is large, the more it becomes difficult to solve algebraically the system.
That is why we restrict ourselves to the case of $n=2$ equations.
We also see during this step that we have to take the derivatives of the coefficients of the equations, which shows the loss of derivatives.
When $n>2$, the algebraic solvability becomes a difficult task that involves the same arguments as in \cite{CL} to be solved.
This is not the purpose of the present paper but this could be the investigation of further developments.
On the other hand, once the algebraic solvability is established (under some conditions), the rest of the proof of Theorem \ref{thm quasilin} below remains unchanged whether $n=2,3,\ldots$

Our main result is the following local exact controllability result around the trajectory $(0,0)$.
\begin{theorem}\label{thm quasilin}
Assume that
\begin{equation}\label{hyp temps quasilin}
T>\left(L-(b-a)\right) \max \ens{\frac{1}{\abs{\lambda_1(0)}},\frac{1}{\abs{\lambda_2(0)}}},
\end{equation}
and
\begin{equation}\label{hyp f}
\frac{\partial f_2}{\partial y_1}(0) \neq 0.
\end{equation}
Then, for every $\epsilon>0$, there exists $\mu>0$, such that, for every $y^0,y^1 \in C^6_L([0,L])^2$ that satisfy
$$\norm{y^0}_{C^6} \leq \mu, \quad \norm{y^1}_{C^6} \leq \mu,$$
there exists a control $\cont \in C^1_L([0,T]\times[0,L])$  that satisfy
\begin{subequations}
\begin{gather}
\supp \cont \subset [\delta,T-\delta]\times[a+\delta,b-\delta], \label{reduc supp} \\
\norm{\cont}_{C^1} \leq \epsilon,
\end{gather}
\end{subequations}
for every $0<\delta<\min(T,(b-a)/2)/4$ such that
\begin{equation}\label{cond delta}
T-4\delta>\left(L-(b-a-8\delta)\right) \max \ens{\frac{1}{\abs{\lambda_1(0)}},\frac{1}{\abs{\lambda_2(0)}}},
\end{equation}
and such that the corresponding solution $y \in C^1_L([0,T]\times[0,L])^n$ to \eqref{syst quasilin} satisfy
\begin{subequations}
\begin{gather}
y(T,x)=y^1(x), \quad \forall x \in [0,L],\\
\norm{y}_{C^1} \leq \epsilon.
\end{gather}
\end{subequations}
\end{theorem}

\begin{remark}
Recall that, given $C^1$ data $y^0$ and $\cont$, \eqref{syst quasilin} has a $C^1$ solution.
Now observe that in Theorem \ref{thm quasilin} the initial and final data are smoother than the control.
This gap between the regularities is not a technical matter (even though the regularity $C^6$ can probably be weakened).
Indeed, consider the linear system
$$\left\{
\begin{array}{l}
u_t + u_x=\cont, \quad (t,x) \in [0,T]\times[0,L], \\
v_t + \lambda v_x +u =0, \quad (t,x) \in [0,T]\times[0,L], \\
u(t,L)=u(t,0), \quad v(t,L)=v(t,0), \quad t \in [0,T], \\
u(0,x)= u^0(x), \quad v(0,x)=v^0(x), \quad x \in [0,L],
\end{array}
\right.
$$
with $\lambda>1$.
Then, for $L-(b-a)<T<L$, writing the system along the characteristics $(t,t+\lambda x)$, we obtain the relation
\begin{multline*}
\int_0^T \cont(t,t+\lambda x) \, dt=u^1(T+\lambda x)-u^0(\lambda x) \\
-(1-\lambda)\lambda\left((v^1)'(T+\lambda x)-(v^0)'(\lambda x)\right), \quad \forall x \in [0,(L-T)/\lambda].
\end{multline*}
This shows that, acting by $C^1$ controls requires $C^2$ initial and final data.
\end{remark}

\subsection{Controllability by two controls}

The starting point of the proof of Theorem \ref{thm quasilin} is to control System \eqref{syst quasilin} with 2 internal controls.
We are going to use the results of \cite{LR2003} on the controllability of $n \times n$ quasilinear systems by $n$ boundary controls and an extension method to obtain the following result.
Observe the different levels of regularity between the state and the controls.

\begin{proposition}\label{prop n cont}
Let us consider the system
\begin{equation}\label{syst star}
\left\{
\begin{array}{l}
y_t + \Lambda(y)y_x + f(y)=e_1\cont_1+e_2\cont_2, \quad (t,x) \in [0,T]\times[0,L], \\
y(t,L)=y(t,0), \quad t \in [0,T], \\
y(0,x)= y^0(x), \quad x \in [0,L].
\end{array}
\right.
\end{equation}
Assume that \eqref{hyp temps quasilin} holds and let \eqref{cond delta} be satisfied for $\delta/2$ (in place of $\delta$).
Then, for every $\epsilon>0$, there exists $\mu>0$ such that, for every $y^0,y^1 \in C^k_L([0,L])^2$ ($k \geq 2$) that satisfy
$$\norm{y^0}_{C^k} \leq \mu, \quad \norm{y^1}_{C^k} \leq \mu,$$
there exist controls $\cont_1,\cont_2 \in C^{k-1}_L([0,T]\times[0,L])$ that satisfy \eqref{reduc supp} and
$$\norm{\cont_1}_{C^{k-1}}+\norm{\cont_2}_{C^{k-1}} \leq \epsilon,$$
and such that the corresponding solution $y \in C^{k}_L([0,T]\times[0,L])^n$ to \eqref{syst star} satisfies
\begin{subequations}
\begin{gather}
y(T,x)=y^1(x), \quad \forall x \in [0,L], \label{donnee finale}\\
\norm{y}_{C^{k}} \leq \epsilon.
\end{gather}
\end{subequations}
\end{proposition}

\begin{proof}
According to \eqref{ordre lambda} we can always assume that $\lambda_1(y)<0<\lambda_2(y)$ for every $y \in \R^2$, the two other cases being similar. Let us then consider the following boundary control problem on the domain $[0,T]\times[b-\delta,a+\delta+L]$:
\begin{equation}\label{syst bcontrols}
\left\{
\begin{array}{l}
y^*_t + \Lambda(y^*)y^*_x + f(y^*)=0, \quad (t,x) \in [0,T]\times[b-\delta,a+\delta+L], \\
y^*_1(t,a+\delta+L)=H_1(t),\quad y^*_2(t,b-\delta)=H_2(t), \quad t \in [0,T], \\
y^*(0,x)= \ext{y^0}(x), \quad x \in [b-\delta,a+\delta+L], \\
y^*(T,x)= \ext{y^1}(x), \quad x \in [b-\delta,a+\delta+L], \\
\end{array}
\right.
\end{equation}
where $H_1,H_2 \in C^k([0,T])$ are boundary controls, and where we have extended by periodicity $y^0$ and $y^1$ to
$$
\ext{y^0}(x) \eqdef
\left\{\begin{array}{cl}
y^0(x) & \mbox{ if } x \in [b-\delta,L], \\
y^0(x-L) & \mbox{ if } x \in [L,a+\delta+L],
\end{array}\right.
$$
$$
\ext{y^1}(x)\eqdef
\left\{\begin{array}{cl}
y^1(x) & \mbox{ if } x \in [b-\delta,L], \\
y^1(x-L) & \mbox{ if } x \in [L,a+\delta+L].
\end{array}\right.
$$
Note that, since $y_0,y_1 \in C^k_L([0,L])^n$, we have $\ext{y^0},\ext{y^1} \in C^k([b-\delta,a+\delta+L])^2$.
Since, by assumption, $T$ satisfies
$$
T>\left(L-(b-a-2\delta)\right) \max \ens{\frac{1}{\abs{\lambda_1(0)}},\frac{1}{\abs{\lambda_2(0)}}},
$$
by \cite[Theorem 1.2]{LR2003}, for every $\mu>0$ small enough, for every $y^0, y^1$ such that
$$\norm{y^0}_{C^k} \leq \mu, \quad \norm{y^1}_{C^k} \leq \mu,$$
there exists $y^* \in C^k([0,T]\times[b-\delta,a+\delta+L])^2$ that satisfies \eqref{syst bcontrols} with $\norm{y^*}_{C^k}$ small.
Let $y^{**}$ be any $C^k([0,T]\times[0,L])^2$ function such that $\norm{y^{**}}_{C^k} \leq \norm{y^*}_{C^k}$ and
$$
y^{**}(t,x)=
\left\{\begin{array}{cl}
y^*(t,x+L) & \mbox{ if } x \in [0,a+\delta], \\
y^*(t,x) & \mbox{ if } x \in [b-\delta,L].
\end{array}\right.
$$
On the other hand, let us introduce $\ext{y} \in C^k([0,T]\times[0,L])^2$ defined by
$$\ext{y}(t,x) \eqdef \eta_1(t)u(t,x)+\eta_2(t)v(t,x)$$
where $\eta_1,\eta_2 \in C^{\infty}([0,T])$ are time cut-off functions with $0 \leq \eta_i \leq 1$ and
\begin{subequations}
\begin{gather}
\eta_1(0)=1, \quad \eta_1(T)=0, \quad \eta_1^{(i)}(0)=\eta_1^{(i)}(T)=0,\quad \forall i \in \inter{1,k+1},\\
\eta_2(0)=0, \quad \eta_2(T)=1, \quad \eta_2^{(i)}(0)=\eta_2^{(i)}(T)=0,\quad \forall i \in \inter{1,k+1},
\end{gather}
\end{subequations}
and $u,v \in C^k([0,T]\times[0,L])^2$ are the solutions to the forward and backward free evolving systems
$$
\left\{
\begin{array}{l}
u_t + \Lambda(u)u_x + f(u)=0, \quad (t,x) \in [0,T]\times[0,L], \\
u(t,L)=u(t,0), \quad t \in [0,T], \\
u(0,x)= y^0(x), \quad x \in [0,L],
\end{array}
\right.
$$
and
$$
\left\{
\begin{array}{l}
v_t + \Lambda(v)v_x + f(v)=0, \quad (t,x) \in [0,T]\times[0,L], \\
v(t,L)=v(t,0), \quad t \in [0,T], \\
v(T,x)= y^1(x), \quad x \in [0,L].
\end{array}
\right.
$$
Let now $\xi \in C^{\infty}([0,L])$ be a space cut-off function with $0\leq \xi \leq 1$ and
$$
\xi(x)=\left\{\begin{array}{rl}
1 & \mbox{ if } x \in [0,a] \cup [b,L], \\
0 & \mbox{ if } x \in [a+\delta,b-\delta]. \\
\end{array}\right.
$$
Let $y$ and $\cont$ be defined by
$$
y(t,x) \eqdef \xi(x)y^{**}(t,x)+(1-\xi(x))\ext{y}(t,x),
$$
and
\begin{equation}\label{def contstar}
\cont \eqdef
y_t + \Lambda(y)y_x + f(y).
\end{equation}
By construction, $y \in C^k([0,T]\times[0,L])^2$ and $\cont \in C^{k-1}([0,T]\times[0,L])^2$.
Still by construction, $(y,\cont)$ solves \eqref{syst star}, $y$ satisfies \eqref{donnee finale} and $\cont$ satisfies
\begin{subequations}
\begin{gather}
\supp \cont \subset [0,T]\times[a,b],\\
\partial^i_t\cont(0,\cdot)=\partial^i_t \cont(T,\cdot)=0,\quad \forall i\in \inter{0,k+1}.
\end{gather}
\end{subequations}
The smallnesses of $y$ and $\cont$ follow from the smallnesses of $y^*$ and $\ext{y}$.

To obtain \eqref{reduc supp} we let system \eqref{syst star} evolves freely (without control), forward on the domain $[0,\delta]\times[0,L]$ and backward on the domain $[T-\delta,T]\times[0,L]$, and denote by $y^\delta(x)$ (\textit{resp.} $y^{T-\delta}(x)$) its value at time $t=\delta$ (\textit{resp.} $t=T-\delta$).
By the previous step, replacing $[a,b]$ by  
$[a+\delta, b-\delta]$, there exists a control $\cont$ with \eqref{reduc supp} and
$$\partial^i_t\cont(\delta,\cdot)=\partial^i_t \cont(T-\delta,\cdot)=0,\quad \forall i\in \inter{0,k+1},$$
that steers the solution to System \eqref{syst star}, posed on the time reduced domain $[\delta,T-\delta]\times[0,L]$, from $y^{\delta}$ to $y^{T-\delta}$.
Thus, we can extend $\cont$ by zero outside $(\delta,T-\delta)\times(a+\delta,b-\delta)$.
Note that the smallnesses are preserved.
\end{proof}

\subsection{Algebraic solvability}\label{sect alg}

In this section we recall the notion of algebraic solvability and the fixed point theorem of M.~Gromov.
We refer to \cite[Section 2.3]{Gro} for more details.

In what follows $Q$ is a smooth bounded open subset of $\R^2$ and $\opD: C^\rpourregularite (\ext{Q})^p \longrightarrow C^0(\ext{Q})^q$ ($p,q \in \N^*$) is a nonlinear $C^\infty$-differential operator of order $\rpourregularite  \in \N^*$.
We recall that this means that there exists a $C^{\infty}$-function $F:\R^{n_{\rpourregularite ,p}} \longrightarrow \R^q$, where $n_{\rpourregularite ,p}=2+p \card \ens{(\alpha_1,\alpha_2) \in \N^2 \, \middle| \, \alpha_1+\alpha_2 \leq \rpourregularite }$, such that $\opD$ writes
$$\opD(z)=F(J^\rpourregularite  z), \quad \forall z \in C^\rpourregularite (\ext{Q})^p,$$
where $J^\rpourregularite  z$ denotes the $\rpourregularite $-jet of $z$, that is the function defined for every $(t,x) \in \ext{Q}$ by
$$J^\rpourregularite  z(t,x) \eqdef\left((t,x),z(t,x),\ldots,\frac{\partial^{\abs{\alpha}} z}{{\partial t}^{\alpha_1} {\partial x}^{\alpha_2} } (t,x), \ldots, \frac{\partial^\rpourregularite  z}{{\partial t}^{\alpha_1} {\partial x}^{\alpha_2} } (t,x) \right)
\in \R^{n_{\rpourregularite ,p}}.$$
Clearly, the map $\opD$ is of class $C^{\infty}$ and we denote by $\opL_{z}:C^\rpourregularite (\ext{Q})^p \longrightarrow C^0(\ext{Q})^q$ its total differential at  $z \in C^\rpourregularite (\ext{Q})^p$.

\begin{definition}\label{def difrel}
We say that $\setA$ is a differential relation of order $d$ ($d \in \N$) if there exists $\setR \subset \R^{n_{d,p}}$ such that
$$
\setA \eqdef
\ens{z \in C^d(\ext{Q})^p \, \middle| \, J^d z(t,x) \in \setR, \quad \forall (t,x) \in \ext{Q}}.
$$
It is said to be open if it is an open subset of $C^d(\ext{Q})^p$.
\end{definition}

\begin{definition}\label{def alg}
Let $\setA \subset C^d(\ext{Q})^p$ be a differential relation of order $d$.
We say that the operator $\opD$ admits an infinitesimal inversion of order $\spourregularite   \in \N$ over $\setA$ if there exists a family of linear differential operators of order $\spourregularite  $, $\opM_z: C^\spourregularite  (\ext{Q})^q \longrightarrow C^0(\ext{Q})^p$, $z \in \setA$, such that:
\begin{enumerate}
\item
For every $g \in C^\spourregularite  (\ext{Q})^q$ being fixed, $z \mapsto \opM_{z}(g)$ is a differential operator of order $d$ (possibly nonlinear) and it is a $C^{\infty}$-differential operator in $(z,g)$.
\item
Algebraic solvability:
for every $z \in \setA^{d+\spourregularite  }\eqdef \setA \cap C^{d+\spourregularite  }(\ext{Q})^p$, we have
$$\opL_{z} \circ \opM_{z}=\Id_{C^{\rpourregularite +\spourregularite  }(\ext{Q})^q},$$
\end{enumerate}
\end{definition}

The proof of Theorem \ref{thm quasilin} is based on the following result \cite[Section 2.3.2, Main Theorem]{Gro}:

\begin{theorem}\label{thm Gromov}
Let $\setA \subset C^d(\ext{Q})^p$ be a nonempty open differential relation of order $d$.
Assume that $\opD$ admits an infinitesimal inversion of order $\spourregularite  $ over $\setA$.
Let
\begin{gather}
\label{propsigma0}
\sigma_0>\max(d,2\rpourregularite +\spourregularite  ),
\\
\label{propnu}
\nu\in (0,+\infty).
\end{gather}
Then, there exist a family of sets $\setB_{z} \subset C^{\sigma_0+\spourregularite  }(\ext{Q})^q$ and a family of operators $\opD^{-1}_{z}: \setB_{z} \longrightarrow \setA$, where ${z} \in \setA^{\sigma_0+\rpourregularite +\spourregularite  }$, such that the following properties hold.
\begin{enumerate}
\item
Neighborhood property:
for every $z \in \setA^{\sigma_0+\rpourregularite +\spourregularite  }$, we have $0 \in \setB_{z}$ and the set $\setB$ defined by
$$\setB= \bigcup_{z \in \setA^{\sigma_0+\rpourregularite +\spourregularite  }} \ens{z} \times \setB_{z},$$
is an open subset of $C^{\sigma_0+\rpourregularite +\spourregularite  }(\ext{Q})^p \times C^{\sigma_0+\spourregularite  }(\ext{Q})^q$.
\item
Inversion property:
\begin{equation}\label{inv prop}
\opD\left(\opD^{-1}_{z}(g)\right)=\opD(z)+g, \quad \forall \left(z,g\right) \in \setB.
\end{equation}
\item
Normalization property:
$$\opD^{-1}_{z}(0)=z, \quad \forall z \in \setA^{\sigma_0+\rpourregularite +\spourregularite  }.$$
\item
Locality:
for every $(t,x) \in \ext{Q}$ and for every $\left(z_1,g_1\right)$, $\left(z_2,g_2\right) \in \setB$, if we have
$$\left(z_1,g_1\right)(\tilde{t},\tilde{x})=\left(z_2,g_2\right)(\tilde{t},\tilde{x}), \quad \forall (\tilde{t},\tilde{x}) \in \ball((t,x),\nu) \cap \ext{Q},$$
then,
$$\opD^{-1}_{z_1}(g_1)(t,x)=\opD^{-1}_{z_2}(g_2)(t,x).$$
\end{enumerate}
\end{theorem}

\subsection{Proof of Theorem \ref{thm quasilin}}

All along this section, $\delta>0$ is a fixed real number such that \eqref{cond delta} holds.
Let us denote
$$Q_\delta \eqdef (\delta,T-\delta)\times(a+\delta,b-\delta),$$
and let $Q$ be an open subset of $\R \times \R$ of class $C^\infty$ such that (see Fig. \ref{fig match})
$$\ext{Q_\delta} \subset Q, \quad \ext{Q}\subset (0,T)\times(a,b).$$
\begin{figure}[h!]
\centering
\begin{tikzpicture}[scale=0.35]
\draw[ultra thick,color=black,fill=white]
  plot[smooth cycle]
  coordinates{
   (-2,-2) (-2,6) (-1.5,18) (8,17) (14,17) (12.5,6) (14,-1)
  };

\draw (-4,-4) rectangle (16,19);
\draw[dashed,fill=gray!40] (0,0) rectangle (12,15);
\draw[dashed,fill=white] (2,2) rectangle (10,13);
\draw (4,4) rectangle (8,11);

\draw[thick] (11,9) circle (2.3);
\draw (11,9) node {$\bullet$};
\draw (11.15,8.5) node[scale=1.15] {$\ext{y}=y^{*}$};

\draw [thick] [<->] (10.3,9.6)--(11.7,9.6);

\draw (6,5) node[scale=1.5] {$Q_{2\delta}$};
\draw (4.25,3) node[scale=1.5] {$Q_{3\delta/2}$};
\draw (2,1) node[scale=1.5] {$Q_{\delta}$};
\draw (0,-1.5) node[scale=1.5] {$Q$};

\end{tikzpicture}
\caption{Matching $\ext{y}$ to $y^{*}$.}
\label{fig match}
\end{figure}

Let us introduce the operator $\opD$ defined by
$$
\begin{array}{rccl}
\opD: & C^1(\ext{Q})^3 & \longrightarrow & C^0(\ext{Q})^2 \\
& (y,\cont) &\longmapsto& y_t+\Lambda(y)y_x+f(y)-e_1\cont.
\end{array}
$$
This is a nonlinear $C^{\infty}$-differential operator of order 1.
Note that $(y,\cont)$ solves the equation of \eqref{syst quasilin} if, and only if,
$$\opD(y,\cont)=0.$$
Let $\opL_{\left(\tilde{y},\tilde{\cont}\right)}$ be the differential of the operator $\opD$ at $(\tilde{y},\tilde{\cont}) \in C^1(\ext{Q})^3$:
$$
\begin{array}{rccl}
\opL_{\left(\tilde{y},\tilde{\cont}\right)}: & C^1(\ext{Q})^3 & \longrightarrow & C^0(\ext{Q})^2 \\
& (y,\cont) &\longmapsto& y_t+\Lambda(\tilde{y})y_x+\left(\Lambda'(\tilde{y})y\right)\tilde{y}_x+f'(\tilde{y})y -e_1\cont.
\end{array}
$$
This is a nonlinear $C^{\infty}$-differential operator of order 1 in $\left(\tilde{y},\tilde{\cont}\right)$.

Let us denote
\begin{gather}
y=\begin{pmatrix}u\\v\end{pmatrix},
\quad
\tilde{y}=\begin{pmatrix}\tilde u\\\tilde v\end{pmatrix},
\\
\left(\Lambda'(\tilde{y})y\right)\tilde{y}_x+f'(\tilde{y})y=\begin{pmatrix} a_{11} & a_{12} \\ a_{21} & a_{22} \\
\end{pmatrix}\begin{pmatrix}u\\v\end{pmatrix}.
\end{gather}
We have
\begin{gather}\label{a12}
a_{11}=\frac{\partial \lambda_1}{\partial u}(\tilde{y})\frac{\partial \tilde{u}}{\partial x}
+\frac{\partial f_1}{\partial u}(\tilde{y})
,
\quad
a_{21}=\frac{\partial \lambda_2}{\partial u}(\tilde{y})\frac{\partial \tilde{u}}{\partial x}
+\frac{\partial f_2}{\partial u}(\tilde{y}).
\end{gather}

\begin{proposition}\label{prop inf inv}
Let $\setA$ be defined by
\begin{equation}\label{diff rela}
\setA \eqdef
\ens{\left(\tilde{y},\tilde{\cont}\right) \in C^{2}(\ext{Q})^3 \, \middle| \,
a_{21}(t,x) \neq 0, \quad \forall (t,x) \in \ext{Q}}.
\end{equation}
Then, $\setA$ is an open differential relation of order $2$ and the operator $\opD$ admits an infinitesimal inversion of order $1$ over $\setA$.
\end{proposition}
\begin{proof}
Clearly, $\setA$ is a differential relation, associated with the set $\setR$ consisting of the points
$$\left((t,x),\begin{pmatrix} \alpha_{00} \\ \beta_{00} \\ \gamma_{00}, \end{pmatrix}, \begin{pmatrix} \alpha_{10} \\ \beta_{10} \\ \gamma_{10}, \end{pmatrix},\begin{pmatrix} \alpha_{01} \\ \beta_{01} \\ \gamma_{01}, \end{pmatrix},\begin{pmatrix} \alpha_{20} \\ \beta_{20} \\ \gamma_{20}, \end{pmatrix},\begin{pmatrix} \alpha_{11} \\ \beta_{11} \\ \gamma_{11}, \end{pmatrix},\begin{pmatrix} \alpha_{02} \\ \beta_{02} \\ \gamma_{02}, \end{pmatrix}\right)
\in \ext{Q} \times \R^{3 \times 6},$$
such that
$$\frac{\partial \lambda_2}{\partial u} \begin{pmatrix}\alpha_{00} \\ \beta_{00} \end{pmatrix} \alpha_{01}
+\frac{\partial f_2}{\partial u}\begin{pmatrix} \alpha_{00} \\ \beta_{00} \end{pmatrix} \neq 0.$$
We can check that $\setA$ is open in $C^{2}(\ext{Q})^3$.

Let now $\left(\tilde{y},\tilde{\cont}\right) \in \setA$.
Let $(g_1,g_2) \in C^1(\ext{Q})^2$.
We have to solve the equation $\opL_{\left(\tilde{y},\tilde{\cont}\right)}(y,\cont)=(g_1,g_2)$ in such a way that $(y,\cont)$ is a linear combination of derivatives of $g_1$ and $g_2$.
 The equation $\opL_{\left(\tilde{y},\tilde{\cont}\right)}(y,\cont)=(g_1,g_2)$  rewrites as
$$
\left\{\begin{array}{lr}
u_t+\lambda_1(\tilde{y})u_x+ a_{11}u+a_{12}v-\cont &=g_1, \quad \mbox{ in } \ext{Q}, \\
v_t+\lambda_2(\tilde{y})v_x+ a_{21}u+a_{22}v &=g_2, \quad \mbox{ in } \ext{Q}.
\end{array}\right.
$$
By definition of $\setA$, we have
\begin{equation}\label{coeff inv}
a_{21}(t,x) \neq 0, \quad \forall (t,x) \in \ext{Q}.
\end{equation}
In this case, the algebraic solvability is not very difficult: we first put
\begin{equation}\label{def v}
v=0,
\end{equation}
so that the second equation simply becomes $a_{21}u=g_2$.
As a result, using \eqref{coeff inv} we can take
\begin{equation}\label{def u}
u=\frac{1}{a_{21}} g_2.
\end{equation}
 Finally, it remains to set
\begin{equation}\label{def cont}
\cont=-g_1 +\frac{1}{a_{21}} \left((g_2)_t+\lambda_1(\tilde{y})(g_2)_x\right)
+\left(\left(\frac{1}{a_{21}}\right)_t+\lambda_1(\tilde{y})\left(\frac{1}{a_{21}}\right)_x+ \frac{a_{11}}{a_{21}}\right)g_2.
\end{equation}
(Note that \eqref{def cont}, together with \eqref{a12}, shows why $ C^{2}(\ext{Q})^3$ cannot be replaced by $C^{1}(\ext{Q})^3$ in Proposition \ref{prop inf inv}.)

Then, the following family of differential operators satisfies all the required properties ($s=1$) to be an infinitesimal inversion of $\opD$ over $\setA$:
$$
\begin{array}{rccl}
\opM_{\left(\tilde{y},\tilde{\cont}\right)} : & C^1(\ext{Q})^2 & \longrightarrow & C^0(\ext{Q})^3 \\
& (g_1,g_2) &\longmapsto& ((u,v),\cont),
\end{array}
$$
where $u,v$ and $\cont$ are respectively defined by \eqref{def u}, \eqref{def v} and \eqref{def cont}.
\end{proof}

{\em Proof of Theorem \ref{thm quasilin}.}
Let $\epsilon>0$ be fixed.
Let $\setA$ be the open differential relation of order $2$ defined by \eqref{diff rela} and set
\begin{gather}
\label{defAepsilon}
\setA_\epsilon \eqdef \setA \cap \left\{(\tilde{y},\tilde{\cont}) \in C^1(\ext{Q})^3 \quad \middle| \quad 
\norm{\tilde y}_{C^1} < \varepsilon, \, \norm{\tilde{\Theta}}_{C^1}< \varepsilon\right\}.
\end{gather}
Note that $\setA_\epsilon$ is an open differential relation of order 2 which is nonempty since $0 \in \setA_\epsilon$ by assumption \eqref{hyp f}.
We choose
\begin{gather}\label{defnu}
  \nu\eqdef \frac{\delta}{2},
  \\
 d=2,\, \rpourregularite=1,\, \spourregularite=1,\,   \sigma_0=4.
  \label{defsigma0}
\end{gather}
Note that, from \eqref{defnu}, one gets \eqref{propnu} and, from \eqref{defsigma0}, one gets \eqref{propsigma0}.

Thanks to Proposition \ref{prop inf inv}, $\opD$ admits an infinitesimal inversion of order $1$ over $\setA_\epsilon$.
Thus, we can apply Theorem \ref{thm Gromov}  which provides a family of sets $\setB_{\left(\tilde{y},\tilde{\cont}\right)} \subset C^{5}(\ext{Q})^2$ and a family of operators $\opD^{-1}_{\left(\tilde{y},\tilde{\cont}\right)}: \setB_{\left(\tilde{y},\tilde{\cont}\right)} \longrightarrow \setA_\epsilon$, where ${\left(\tilde{y},\tilde{\cont}\right)} \in \setA_\epsilon^{6}\eqdef\setA_\epsilon \cap C^{6}(\ext{Q})^3$, such that all the properties listed in this theorem hold.

Since $0 \in \setA_\epsilon$, $\setB$ also contains $0$.
Since $\setB$ is open, there exists $\rho>0$ such that
\begin{equation}\label{Bcontientvoisinage}
\left\{\left(\left(\tilde y, \tilde{\Theta}\right), g\right)\in C^6(\ext{Q})^3\times C^5(\ext{Q})^2 \quad \middle| \quad 
\norm{\tilde y}_{C^6} \leq \rho, \, \norm{\tilde{\Theta}}_{C^6}\leq \rho, \,
\norm{g}_{C^5}\leq \rho   \right\}\subset \setB.
\end{equation}
On the other hand, by Proposition \ref{prop n cont}, there exists $\mu>0$ such that, for every $y^0,y^1
\in C^{6}_L([0,L])^2$ that satisfy
$$\norm{y^0}_{C^{6}} \leq \mu, \quad \norm{y^1}_{C^{6}} \leq \mu,$$
there exist $\cont^*_1,\cont^*_2 \in C^{5}([0,T]\times[0,L])$ and $y^* \in C^{6}([0,T]\times[0,L])^2$ such that
$$\opD(y^*,0)=-e_1 \cont_1^*-e_2 \cont_2^*,$$
$$y^*(T,x)=y^1(x), \quad \forall x \in [0,L],$$
$$\supp \cont^*_1 \subset Q_{2\delta}, \quad \supp \cont^*_2 \subset Q_{2\delta},$$
and
$$\norm{y^*}_{C^{6}} < \epsilon', \quad \norm{\cont^*_1}_{C^{5}}+\norm{\cont^*_2}_{C^{5}} < \epsilon',$$
where $\epsilon'=\min(\epsilon,\rho)$.
In particular,
$$\left(y^*,0,-e_1\cont^*_1-e_2\cont^*_2\right) \in \setB,$$
and we can set
$$(\ext{y},\ext{\cont}) \eqdef \opD^{-1}_{(y^*,0)}\left(-e_1\cont^*_1-e_2\cont^*_2\right).$$
By definition $(\ext{y},\ext{\cont}) \in C^1(\ext{Q})^3$ and satisfies
$$\opD(\ext{y},\ext{\cont})=\opD(y^*,0)-e_1\cont^*_1-e_2\cont^*_2=0.$$

Let us now prove that we can match $\ext{y}$ to $y^*$ and $\ext{\cont}=0$ in the open neighborhood $Q_{\delta} \backslash \adh{Q}_{3\delta/2}$ (see Fig. \ref{fig match} below).
Let then $(t,x) \in Q_{\delta} \backslash \adh{Q}_{3\delta/2}$ be fixed and let us show that
\begin{equation}\label{equality}
\ext{y}(t,x)=y^*(t,x), \quad \ext{\cont}(t,x)=0.
\end{equation}
Since $\cont^*_1$ and $\cont^*_2$ are supported in $\adh{Q}_{2\delta}$, we have
$$\left(y^*,0,-e_1\cont^*_1-e_2\cont^*_2\right)(\tilde{t},\tilde{x})=\left(y^*,0,0\right)(\tilde{t},\tilde{x}), \quad \forall (\tilde{t},\tilde{x}) \in \ball((t,x),\delta/2),$$
and $\left(y^*,0,-e_1\cont^*_1-e_2\cont^*_2\right), \left(y^*,0,0\right) \in \setB$.
Thus, by \eqref{defnu}, the locality and normalization properties, we obtain \eqref{equality}.
To conclude the proof, it remains to set
$$
y(t,x) \eqdef \left\{\begin{array}{cl}
\ext{y}(t,x) & \mbox{ if } (t,x) \in Q_{\delta}, \\
y^*(t,x) & \mbox{ if } (t,x) \in [0,T]\times [0,L] \backslash Q_{\delta}, \\
\end{array}\right.
$$
and
$$
\cont(t,x) \eqdef \left\{\begin{array}{cl}
\ext{\cont}(t,x) & \mbox{ if } (t,x) \in Q_{\delta}, \\
0 & \mbox{ if } (t,x) \in [0,T]\times [0,L] \backslash Q_{\delta}. \\
\end{array}\right.
$$
Then, $(y,\cont)$ solves \eqref{syst quasilin}.
Since $y=y^*$ near the boundary of $[0,T]\times[0,L]$ it satisfies the same periodic boundary conditions $y(\cdot,L)=y(\cdot,0)$, the same initial condition $y(0,\cdot)=y^0$ and the same final condition $y(T,\cdot)=y^1$.
Finally, note that the smallnesses of $y$ and $\cont$ follow from the smallness of $y^*$ and the definition
\eqref{defAepsilon} of $\setA_\epsilon$.
\endproof

\Appendix

\section{Proof of Lemma \ref{key lemma}}
\label{secproolkeylemma}

Let us compute $\tau_k(x,b)-\tau_k(x,a)$ for every $k \in \inter{\kmin(x,b),\kmax(x,a)}$ and for every $x \in [0,L)$:
\begin{multline*}
\tau_k(x,b)-\tau_k(x,a)= \\
\left\{\begin{array}{cll}
b-x+kL & \mbox{ if } a-x+kL \in (-\infty,0) & \mbox{ and } b-x+kL \in [0,T], \\
T & \mbox{ if } a-x+kL \in (-\infty,0) & \mbox{ and } b-x+kL \in (T,+\infty), \\
b-a & \mbox{ if } a-x+kL \in [0,T] & \mbox{ and } b-x+kL \in [0,T], \\
T-a+x-kL & \mbox{ if } a-x+kL \in [0,T] & \mbox{ and } b-x+kL \in (T,+\infty).
\end{array}\right.
\end{multline*}
Thus, we see that the main problem that we have to handle is the fact that the length $\tau_0(x,b)-\tau_0(x,a)$ (\textit{resp.} $\tau_{\ceil{\frac{T-a}{L}}}(x,b)-\tau_{\ceil{\frac{T-a}{L}}}(x,a)$) goes to zero as $x$ goes to $b^-$ (\textit{resp.} $p(a)^+$).
We are going to approximate uniformly the Gramian associated with \eqref{equ caract}-\eqref{constr ODE} by modifying it a little bit near the points $b$ and $p(a)$ in order to avoid the aforementioned problems.
This has to be done in such a way that the invertibility is preserved.
Let
$$\ell_0=\min\left(T-a,b-a,L-p(a),b-L+\ceil{\frac{T-a}{L}}L\right)/2,$$
(note that $\ell_0$ is positive thanks to \ref{hyp temps}).
Then, for every $x \in [0,L)$ and $0 \leq \ell <\ell_0$ we introduce the approximated Gramian
$$
Q(x,\ell) =
Q_0(x,\ell)+\sum_{k=1}^{\ceil{\frac{T-a}{L}}-1} Q_k(x,\ell)+Q_{\ceil{\frac{T-a}{L}}}(x,\ell),
$$
where, for every $k \in \inter{1,\ceil{\frac{T-a}{L}}-1}$,
$$Q_k(x,\ell)=\int_{\tau_k(x,a)+\ell}^{\tau_k(x,b)-\ell} G(s,x) \, ds,$$
with $G(s,x)=\Rx\left(T,s\right)\Bx(s) \Bx(s)^*\Rx\left(T,s\right)^*$, and
$$
Q_0(x,\ell)=
\left\{\begin{array}{cl}
\ds \int_{\tau_0(x,a)+\ell}^{\tau_0(x,b)-\ell} G(s,x) \, ds & \mbox{ if } x \in [0,b-2\ell), \\
\ds 0 & \mbox{ if } x \in [b-2\ell,L),
\end{array}
\right.
$$
and where
$$
Q_{\ceil{\frac{T-a}{L}}}(x,\ell)= \\
\left\{\begin{array}{cl}
0 & \mbox{ if } x \in [0,p(a)+2\ell], \\
\ds \int_{\tau_{\ceil{\frac{T-a}{L}}}(x,a)+\ell}^{\tau_{\ceil{\frac{T-a}{L}}}(x,b)-\ell} G(s,x) \, ds & \mbox{ if } x \in (p(a)+2\ell,L). \\
\end{array}
\right.
$$

{\em Step 1:}
Firstly, let us prove that there exists $r>0$ such that, for every $x \in [0,L)$,
\begin{equation}\label{prop inv}
\ball\left(Q(x,0),r\right) \subset \gln,
\end{equation}
where $\gln$ denotes again the set of invertible matrices of size $n$ and $\ball(M,\rho)$ the open ball of center $M \in \R^{n \times n}$ and radius $\rho>0$.
Since $\lim_{\substack{x \to L \\ x<L}} Q(x,0)=Q(0,0)$, we can extend $Q(\cdot,0)$ by continuity to $[0,L]$ (still denoted by $Q(\cdot,0)$) with $Q(L,0)=Q(0,0)$.
By assumption \ref{hyp caract}, for every $x \in [0,L]$, $Q(x,0) \in \gln$.
Since $\gln$ is open, there exists $r(x)>0$ such that
\begin{equation}\label{inv ac x}
\ball\left(Q(x,0),r(x)\right) \subset \gln.
\end{equation}
On the other hand,
$$Q([0,L],0) \subset \bigcup_{x \in [0,L]} \ball\left(Q(x,0),\frac{r(x)}{2}\right),$$
and $Q([0,L],0)$ is compact (by continuity of $Q(\cdot,0)$ on all $[0,L]$), so that there exists $x_1, \ldots, x_q \in [0,L]$ such that
$$
Q([0,L],0) \subset \bigcup_{x \in \ens{x_1, \ldots, x_q}} \ball\left(Q(x,0),\frac{r(x)}{2}\right).
$$
We define
$$r=\min\ens{\frac{r(x_1)}{2},\ldots,\frac{r(x_q)}{2}}.$$
Thus, for every $x \in [0,L]$, for every $M \in \ball\left(Q(x,0),r\right)$, there exists $x_i \in \ens{x_1,\ldots,x_q}$, such that
$$
\norm{M-Q(x_i,0)}\leq
\norm{M-Q(x,0)}+\norm{Q(x,0)-Q(x_i,0)}
<r+\frac{r(x_i)}{2} \leq r(x_i),$$
that is $M \in \ball(Q(x_i,0),r(x_i))$ and shows that $M$ is invertible by \eqref{inv ac x}.

{\em Step 2:}
By construction we have
$$\sup_{x \in [0,L)}\norm{Q(x,\ell)-Q(x,0)} \xrightarrow[\ell \to 0]{} 0.$$
Thus, there exists $\delta>0$ small enough such that, for every $x \in [0,L)$, we have
$$Q(x,2\delta) \in \ball\left(Q(x,0),r\right),$$
which shows that $Q(x,2\delta)$ is invertible by \eqref{prop inv}.

{\em Step 3:}
Let us now define the cut-off function $\eta$ (see Fig. \ref{fig eta}).
Let us introduce
\begin{multline*}
\Tau_0(\delta) \eqdef \\
\ens{(t,x) \in (0,T)\times(0,L) \quad \middle| \quad x \in (0,b-2\delta), \quad t \in \left(\tau_0(x,a)+\delta,\tau_0(x,b)-\delta\right)},
\end{multline*}
\begin{multline*}
\Tau_{\ceil{\frac{T-a}{L}}}(\delta) \eqdef
\Big\{(t,x) \in (0,T)\times \left(0,L\right), \quad \Big| \\
\quad x \in (p(a)+2\delta,L), \quad t \in \left(\tau_{\ceil{\frac{T-a}{L}}}(x,a)+\delta,\tau_{\ceil{\frac{T-a}{L}}}(x,b)-\delta\right)\Big\},
\end{multline*}
and, for $k \in \inter{1,\ceil{\frac{T-a}{L}}-1}$,
$$
\Tau_k(\delta) \eqdef
\ens{(t,x) \in (0,T)\times(0,L) \quad \middle| \quad t \in \left(\tau_k(x,a)+\delta,\tau_k(x,b)-\delta\right)}.
$$
Let $\xi \in C^1([0,T]\times[0,L])$ be a cut-off function with $0 \leq \xi \leq 1$ and such that (see Fig. \ref{fig trapezium} to help)
\begin{subequations}
\begin{gather}
\xi \equiv 1 \mbox{ in } \bigcup_{k=0}^{\ceil{\frac{T-a}{L}}} \Tau_k(2\delta), \label{xivaut11}\\
\xi \equiv 1 \mbox{ in } \bigcup_{k=0}^{\ceil{\frac{T-a}{L}}-1} \ens{(t,0) \, \middle| \, t \in (\tau_k(0,a)+2\delta,\tau_k(0,b)-2\delta)}, \label{xivaut12}\\
\xi \equiv 1 \mbox{ in } \bigcup_{k=1}^{\ceil{\frac{T-a}{L}}} \ens{(t,L) \, \middle| \, t \in (\tau_{k-1}(0,a)+2\delta,\tau_{k-1}(0,b)-2\delta)}, \label{xivaut13}\\
\xi \equiv 0 \mbox{ in } [0,T]\times[0,L] \backslash \left(\bigcup_{k=0}^{\ceil{\frac{T-a}{L}}} \ext{\Tau_k(\delta)}\right). \label{xi zero}
\end{gather}
\end{subequations}
For every $x \in [0,L]$, let us define
$$
Q_x= \int_{0}^{T}  G(s,x) \xi(s,x) \, ds.
$$
Note that $Q_L=Q_0$.
Let us show that $Q_x$ is invertible for every $x \in [0,L]$.
First, note that the controllability Gramian $Q_x$ is a nonnegative symmetric matrix.
Thus, it is invertible if, and only, if it is positive definite.
Let $v \in \R^n$.
Using $\xi \geq 0$ and \eqref{xivaut11}-\eqref{xi zero}, we have
$$
\begin{array}{rcl}
v \cdot Q_x v &=&\ds \int_0^T \norm{B(s)^*R(T,s)^*v}^2 \xi(s,x)\, ds \\
& =&\ds
\sum_{k=0}^{\ceil{\frac{T-a}{L}}} \int_{\ens{t \in [0,T] \, \middle| \, (t,x) \in \ext{\Tau_k(\delta)}}} \norm{B(s)^*R(T,s)^*v}^2 \xi(s,x) \, ds \\
& \geq &\ds
\sum_{k=0}^{\ceil{\frac{T-a}{L}}} \int_{\ens{t \in [0,T]  \, \middle| \, (t,x) \in \Tau_k(2\delta)}} \norm{B(s)^*R(T,s)^*v}^2 \xi(s,x) \, ds 
= v \cdot Q(x,2\delta) v.
\end{array}
$$
As a result, the positive definiteness of $Q_x$ follows from the one of $Q(x,2\delta)$.

To conclude, it remains to set, for every $(t,x) \in [0,T]\times[0,L]$,
$$\eta(t,x) \eqdef
\left\{\begin{array}{cl}
\xi(t,\excar^{-1}(t,x)) & \mbox{ if }  (t,x) \in \Tau', \\
\xi(t,0) & \mbox{ if }  x=\excar(t,0), \\
0 & \mbox{ if } (t,x) \in \partial \left([0,T]\times[0,L]\right).
\end{array}
\right.
$$
Thanks to \eqref{xi zero}, we have $\eta \in C^1([0,T]\times[0,L])$ and \eqref{eta zero}.
Moreover,
$$\eta\left(t,\excar(t,x)\right)=\xi(t,x), \quad \forall (t,x) \in [0,T]\times[0,L]. \qquad \endproof$$

\bibliographystyle{amsalpha}
\bibliography{biblio}

\end{document}